\documentclass[reqno]{amsart}
\usepackage{amscd}
\usepackage{amssymb}
\usepackage[margin=1in,includeheadfoot]{geometry}
\usepackage{cite}
\usepackage [latin1]{inputenc}
\usepackage{tabularx,booktabs,tikz}
\usepackage{caption}
\usepackage{amsmath}
\usepackage{amsfonts}

\usepackage{amscd}
\usepackage{amsthm}
\usepackage{amssymb} \usepackage{latexsym}
\usepackage{eufrak}
\usepackage{euscript}
\usepackage{epsfig}
\usepackage{graphics}
\usepackage{array}
\usepackage{enumerate}
\usepackage{dsfont}
\usepackage{color}
\usepackage{wasysym}
\usepackage{hyperref}
\usepackage{pdfsync}

\def\beq{\begin{equation}}
\def\eeq{\end{equation}}
\def\ba{\begin{array}}
\def\ea{\end{array}}

\def\R{\mathbb R}

\newtheorem{thm}{Theorem}[section]
\newtheorem{lm}[thm]{Lemma}

\theoremstyle{definition}
\newtheorem{rem}[thm]{Remark}

\newtheorem{df}[thm]{Definition}

\theoremstyle{remark}

\begin{document}
\pagestyle{plain}
\title{Positive solutions to fractional $p$-Laplacian Choquard equation on lattice graphs}

\author{Lidan Wang}
\email{wanglidan@ujs.edu.cn}
\address{Lidan Wang: School of Mathematical Sciences, Jiangsu University, Zhenjiang 212013, People's Republic of China}

\begin{abstract}
In this paper, we study the fractional $p$-Laplacian Choquard equation 
$$
(-\Delta)_{p}^{s}  u+h(x)|u|^{p-2} u=\left(R_{\alpha} *F(u)\right)f(u)
$$
on lattice graphs $\mathbb{Z}^d$, where $s\in(0,1)$, $ p\geq 2$, $\alpha \in(0, d)$ and $R_\alpha$ represents the Green's function of the discrete fractional Laplacian that behaves as the Riesz potential. Under suitable assumptions on the potential function $h$, we first prove the existence of a strictly positive solution by the mountain-pass theorem for the  nonlinearity $f$ satisfying some growth conditions.  Moreover, if we add some monotonicity condition, we establish the existence of a positive ground state solution by the method of Nehari manifold.
\end{abstract}

\maketitle

{\bf Keywords:} fractional $p$-Laplacian Choquard equation, positive solution, ground state solution, mountain-pass theorem, Nehari manifold 

\
\

{\bf Mathematics Subject Classification 2020:} 35A15, 35R02, 35R11

\section{Introduction}
The nonlinear Choquard equation  
\begin{equation*}
(-\Delta)_{p}^{s}  u+h(x)|u|^{p-2} u=\left(I_{\alpha} *F(u)\right)f(u),\quad x\in\mathbb{R}^d,
\end{equation*}
where $I_\alpha$ is the Riesz potential, has drawn lots of interest in recent years. In particular, for $s=1$ and $p=2$,  this equation turns into the well-known Choquard equation, which has been studied extensively in the literature, see for examples \cite{ANY1,ANY2,GV,L2, MV2,MV3}. For $s\in (0,1)$ and $p=2$, this equation becomes the fractional Choquard equation, we refer the readers to \cite{GH,LS,YZ,YC,ZU,ZYW}. For $s\in(0,1)$ and $p>2$, this equation is the fractional $p$-Laplacian Choquard equation,  which has also been considered, see \cite{MZ,RL,S,S1,ZS}.

Nowadays, there are many works on graphs, see for examples \cite{GLY,HW,HLW,LY,SY,W2,W5,YZ1}. Recently, the discrete nonlinear Choquard equation has attracted much attention. The discrete analog of the above equation is
\begin{equation*}\label{aa}
(-\Delta)_{p}^{s}  u+h(x)|u|^{p-2} u=\left(R_{\alpha} *F(u)\right)f(u),\quad x\in\mathbb{Z}^d, 
\end{equation*}
where $R_\alpha$ behaves as the Riesz potential. For $s=1$ and $p=2$, Wang \cite{WZW} proved the existence of ground state solutions to the Choquard equation with power nonlinearity on lattice graphs. Later, the authors in \cite{LW1, LZ1,LZ2} derived the existence of ground state solutions to the biharmonic Choquard equation on lattice graphs. For $s=1$ and $p>2$, Liu \cite{L8} and Wang \cite{W8} established the existence of  positive solutions  and ground state solutions to this $p$-Laplacain Choquard equation  respectively on lattice graphs. Moreover,
Wang \cite{W4} obtained the existence of ground state solutions to a class of $p$-Laplacian equations with power Choquard nonlinearity on lattice graphs.
 For more related works on discrete Choquard equations, we refer the readers to \cite{L9,W1,W3,W7}.

Very recently, the discrete fractional Laplace operator has been studied by \cite{LR,ZY1}.  By using the definition of fractional Laplace operator in \cite{LR}, Wang \cite{W6} proved the existence and multiplicity of solutions to a discrete fractional Sch\"{o}dinger equation. 
 Zhang, Lin and Yang \cite{ZY1} obtained the multiplicity solutions to a discrete fractional Schr\"{o}dinger equation based on the definition of fractional  Laplace operator introduced by themselves.  Later, Zhang, Lin and Yang \cite{ZY} proposed a fractional $p$-Laplace operator and established the existence of positive solutions and positive ground state solutions to a discrete fractional $p$-Laplacian Schr\"{o}dinger equation. To the best of our knowledge, there are no works about the discrete fractional $p$-laplacian ($p\geq 2$) Choquard equations. Inspired  by the papers above, in this paper, we study the fractional $p$-Laplacian Choquard equation
\begin{equation}\label{aa}
(-\Delta)_{p}^{s}  u+h(x)|u|^{p-2} u=\left(R_{\alpha} *F(u)\right)f(u),\quad x\in\mathbb{Z}^d, 
\end{equation}
 where $s\in(0,1)$, $ p\geq 2$, $\alpha \in(0, d)$ and $R_\alpha$ stands for the Green's function of the discrete fractional Laplacian, $$ R_{\alpha}(x,y)=\frac{K_{\alpha}}{(2\pi)^d}\int_{\mathbb{T}^d}e^{i(x-y)\cdot k}\mu^{-\frac{\alpha}{2}}(k)\,dk,\quad x,y\in \mathbb{Z}^d,$$
which contains the fractional degree
$$K_\alpha=\frac{1}{(2\pi)^d}\int_{\mathbb{T}^d}\mu^{\frac{\alpha}{2}}(k)\,dk,\,\mu(k)=2d-2\sum\limits_{j=1}^d \cos(k_j),$$
where $\mathbb{T}^d=[0,2\pi]^d,\,k=(k_1,k_2,\dots,k_d)\in\mathbb{T}^d,$ and $\frac{2\pi}{d_j}\ell_j\rightarrow k_j$,\,$\ell_j=0,1,\dots,d_j-1,$ as $d_j\rightarrow\infty$.  It follows from \cite{MC} that the Green's function $R_{\alpha}$ behaves as $|x-y|^{-(d-\alpha)}$ for $|x-y|\gg 1$.  Here the fractional $p$-Laplace operator $(-\Delta)_{p}^{s}$ is given by 
$$
(-\Delta)_{p}^{s} u(x)=\frac{1}{2} \sum_{y \in \mathbb{Z}^d, y \neq x} W_{s}(x, y)\left(\left|\nabla^{s} u\right|^{p-2}(x)+\left|\nabla^{s} u\right|^{p-2}(y)\right)(u(x)-u(y)),
$$ where $K_s(x,y)$ is a symmetric positive function satisfying
$$
c_{s,d}|x-y|^{-d-2s}\leq K_s(x,y)\leq C_{s,d}|x-y|^{-d-2s},\quad x\neq y.
$$
This order estimate can be seen in \cite[Theorem 2.4]{W0}.  Now we state the assumptions on the potential $h$ and the nonlinearity $f$:
\begin{itemize}
\item[($h_1$)] for any $x\in\mathbb{Z}^d$, there exists a constant $h_0>0$ such that $h(x) \geq h_0$;
\item[($h_2$)] there exists a point $x_0\in\mathbb{Z}^d$ such that $h(x)\rightarrow\infty$ as $|x-x_0|\rightarrow\infty;$
\item[($f_1$)] $f(0)=0,\,f\in C(\mathbb{R},\mathbb{R})$ and $f(t)=o\left(t^{p-1}\right)$ as $t\rightarrow 0^+$;
\item[($f_2$)] there exist constants $C>0$ and $\tau>\frac{(d+\alpha)p}{2d}$ such that
$$
f(t)\leq C(1+t^{\tau-1}), \quad t>0;
$$ 
\item[($f_3$)] there exists a constant $\theta>p$ such that $$0\leq \theta F(t)=\theta \int_{0}^{t}f(s)\,ds \leq 2f(t)t,\quad t>0;$$
\item[($f_4$)] for any $u\in H_{s,p}\backslash\{0\}$, $$\frac{\int_{\mathbb{Z}^d}(R_\alpha \ast F(tu))f(tu)u\,d\mu}{t^{p-1}}$$ is strictly increasing with respect $t\in (0, \infty)$.
\end{itemize}

Obviously, by $(f_1)$ and $(f_2)$, one has that for any $\varepsilon>0$, there exists $C_\varepsilon>0$ such that
\begin{equation}\label{ad}
f(t)\leq\varepsilon t^{p-1}+C_\varepsilon t^{\tau-1},\quad t>0,
\end{equation}
and thus
\begin{equation}\label{ae}
F(t)\leq\varepsilon t^{p}+C_\varepsilon t^{\tau},\quad t>0.
\end{equation}

Our main results are as follows.
\begin{thm}\label{t2}
 Assume that $(h_1)$, $(h_2)$ and $(f_1)$-$(f_3)$ hold. Then the equation (\ref{aa}) has a strictly positive solution.
\end{thm}

\begin{thm}\label{t1}
Assume that $(h_1)$, $(h_2)$ and $(f_1)$-$(f_4)$ hold. Then the equation (\ref{aa}) has a strictly positive ground state solution.    
\end{thm}

This paper is organized as follows. In Section 2, we state some basic results on graphs. In Section 3, we prove Theorem \ref{t2} by the mountain-pass theorem. In Section 4, we verify Theorem \ref{t1} by the method of Nehari manifold.

\section{Preliminaries} 
In this section, we state some basic results on graphs. Let $G=(V, E)$ be a connected, locally finite graph, where $V$ denotes the vertex set and $E$ denotes the edge set. We call vertices $x$ and $y$ neighbors, denoted by $x \sim y$, if there exists an edge connecting them, i.e. $(x, y) \in E$. For any $x,y\in\mathbb{Z}^d$, the distance $d(x,y)$ is defined as the minimum number of edges connecting $x$ and $y$, namely
$$d(x,y)=\inf\{k:x=x_0\sim\cdots\sim x_k=y\}.$$

In this paper, we consider, the natural discrete model of the Euclidean space, the integer lattice graph.  The $d$-dimensional integer lattice graph, denoted by $\mathbb{Z}^d$, consists of the set of vertices $V=\mathbb{Z}^d$ and the set of edges $E=\{(x,y): x,\,y\in\mathbb{Z}^d,\,\underset {{i=1}}{\overset{d}{\sum}}|x_{i}-y_{i}|=1\}.$
In the sequel, we denote $|x-y|:=d(x,y)$ on the lattice graph $\mathbb{Z}^d$.

Let $C(\mathbb{Z}^d)$ be the set of all functions on $\mathbb{Z}^d$. For $u,v \in C(\mathbb{Z}^d)$ and $s \in(0,1)$, as in \cite{ZY}, we define the fractional gradient form as
$$
\nabla^{s} u \nabla^{s} v(x)=\frac{1}{2} \sum_{y \in \mathbb{Z}^d, y \neq x} W_{s}(x, y)(u(x)-u(y))(v(x)-v(y)),
$$
 where $K_s(x,y)$ is a symmetric positive function satisfying
$$
c_{s,d}|x-y|^{-d-2s}\leq K_s(x,y)\leq C_{s,d}|x-y|^{-d-2s},\quad x\neq y.
$$
Moreover, we write the length of $\nabla^{s} u(x)$ as
$$
\left|\nabla^{s} u\right|(x)=\sqrt{\nabla^{s} u \nabla^{s} u(x)}=\left(\frac{1}{2} \sum_{y \in \mathbb{Z}^d, y \neq x} W_{s}(x, y)(u(x)-u(y))^{2}\right)^{\frac{1}{2}} .
$$
For $p \in[2,\infty)$, we define the fractional $p$-Laplacian of $u\in  C(\mathbb{Z}^d)$ as
$$
(-\Delta)_{p}^{s} u(x)=\frac{1}{2} \sum_{y \in \mathbb{Z}^d, y \neq x} W_{s}(x, y)\left(\left|\nabla^{s} u\right|^{p-2}(x)+\left|\nabla^{s} u\right|^{p-2}(y)\right)(u(x)-u(y)).
$$

The space $\ell^{p}(\mathbb{Z}^d)$ is defined as $
\ell^{p}(\mathbb{Z}^d)=\left\{u \in C(\mathbb{Z}^d):\|u\|_{p}<\infty\right\},
$ where
$$
\|u\|_{p}= \begin{cases}\left(\sum\limits_{x \in \mathbb{Z}^d}|u(x)|^{p}\right)^{\frac{1}{p}}, &  1 \leq p<\infty, \\ \sup\limits_{x \in \mathbb{Z}^d}|u(x)|, & p=\infty.\end{cases}
$$
Let $u\in C(\mathbb{Z}^d)$. For convenience, we always write $$
\int_{\mathbb{Z}^d} u(x)\,d \mu=\sum\limits_{x \in \mathbb{Z}^d} u(x),
$$
where $\mu$ is the counting measure in $\mathbb{Z}^d$. 

For any $s \in(0,1)$ and $p \in[2,\infty)$, we define a fractional Sobolev space
$$
W^{s, p}(\mathbb{Z}^d)=\left\{u \in L^{\infty}(\mathbb{Z}^d): \int_{\mathbb{Z}^d}\left(\left|\nabla^{s} u\right|^{p}+|u|^{p}\right) d \mu<\infty\right\}
$$
with a Sobolev norm
$$
\|u\|_{W^{s, p}}=\left(\int_{\mathbb{Z}^d}\left(\left|\nabla^{s} u\right|^{p}+|u|^{p}\right) d \mu\right)^{\frac{1}{p}}.
$$
By \cite[Theorem 1]{ZY}, one gets that $W^{s, p}(\mathbb{Z}^d)$ is a reflexive Banach space.

Let $C_{c}(\mathbb{Z}^d)$ be the set of all functions on $\mathbb{Z}^d$ with finite support. Let $H_{s, p}(\mathbb{Z}^d)$ ($H_{s,p}$ for brevity) be the completion of $C_{c}(\mathbb{Z}^d)$ with respect to the norm
\begin{equation*}
\|u\|_{H_{s, p}}=\left(\int_{\mathbb{Z}^d}\left(\left|\nabla^{s} u\right|^{p}+h(x)|u|^{p}\right) d \mu\right)^{\frac{1}{p}}. 
\end{equation*}
It is clear that $H_{s, p}$ is a Banach space.  Moreover, by $(h_1)$, we have \begin{equation}\label{lx}
\|u\|_p^p\leq \frac{1}{h_0}\int_{\mathbb{Z}^d} h(x)u^p(x)\,d\mu\leq \frac{1}{h_0}\|u\|_{H_{s, p}}^p,    
\end{equation}
which implies that
$$
\|u\|_{W^{s, p}}^{p} \leq \left(1+\frac{1}{h_{0}}\right)\|u\|_{H_{s, p}}^{p}.
$$
This yields that
$H_{s, p}$ is a closed subspace of $W^{s, p}(\mathbb{Z}^d)$, which means that $H_{s, p}$ is also a reflexive Banach space.

The energy functional $J_{s,p}: H_{s, p}\rightarrow\R$ related to the equation (\ref{aa}) is given by
$$
J_{s,p}(u)=\frac{1}{p}\int_{\mathbb{Z}^d}\left(|\nabla^s u|^{p}+h(x)|u|^{p}\right) \,d\mu-\frac{1}{2}\int_{\mathbb{Z}^d}(R_\alpha \ast F(u))F(u)\,d\mu.
$$
Since $p\geq 2$, by the discrete HLS inequality (Lemma \ref{lm1} below) and the formulas of integration by parts (Lemma \ref{l0} below), one gets easily that $J_{s,p}\in C^{1}(H_{s,p},\mathbb{R})$ and, for $\phi\in H_{s,p}$,
\begin{eqnarray*}
\left\langle J_{s,p}^{\prime}(u), \phi\right\rangle=\int_{\mathbb{Z}^d}\left(|\nabla^s u|^{p-2}\nabla^s u \nabla^s \phi+h(x)|u|^{p-2} u \phi\right) \,d\mu-\int_{\mathbb{Z}^d}(R_\alpha \ast F(u))f(u)\phi\,d\mu.
\end{eqnarray*}

\begin{df}
We say $u \in H_{s,p}$ is a nontrivial weak solution to the equation (\ref{aa}) if $u\neq 0$ is a critical point of the functional $J_{s,p}$, i.e. $J_{s,p}'(u)=0$.
We say that $u\in H_{s,p}$ is a ground state solution to the equation (\ref{aa}) if
$u$ is a nontrivial critical point of the functional $J_{s,p}$ such that
$$
J_{s,p}(u)=\inf\limits_{v\in M_{s,p}} J_{s,p}(v),
$$
where $M_{s,p}=\{v\in H_{s,p}\backslash\{0\}: \langle J_{s,p}'(v),v\rangle=0\}$ is the corresponding Nehari manifold. 	
\end{df}

Next, we give some basic lemmas that useful in this paper. The first one is about the formulas of integration by parts, see \cite{ZY}.

\begin{lm}\label{l0} 
Let $u \in W^{s, p}(\mathbb{Z}^d)$. Then for any $\phi \in C_{c}(\mathbb{Z}^d)$, we have
$$
\int_{\mathbb{Z}^d} \phi(-\Delta)_{p}^{s} u\,d \mu=\int_{\mathbb{Z}^d}\left|\nabla^{s} u\right|^{p-2} \nabla^{s} u \nabla^{s} \phi\,d \mu.
$$
\end{lm}

\begin{rem}
 If $u$ is a weak solution to the equation (\ref{aa}), then by Lemma \ref{l0}, one can get easily that $u$ is a pointwise solution to the equation (\ref{aa}).   
\end{rem}

\begin{lm}\label{i1}
Let $u\in H_{s,p}$. Then we have $u^-,u^+, |u|\in H_{s,p}$, where $u^{+}=\max \{u, 0\}$ and $u^{-}=\min \{u, 0\}.$    
\end{lm}

\begin{proof}
Note that $u=u^{+}+u^{-}$ and  $u^{+} u^{-}=0$. A direct calculation yields that
\begin{align}\label{ar}
\nabla^{s} u^{+} \nabla^{s} u^{-}(x) & =\frac{1}{2} \sum_{y\in\mathbb{Z}^d, y \neq x} W_{s}(x, y)\left(u^{+}(x)-u^{+}(y)\right)\left(u^{-}(x)-u^{-}(y)\right) \nonumber\\
& =-\frac{1}{2} \sum_{y\in\mathbb{Z}^d, y \neq x} W_{s}(x, y)\left(u^{+}(x) u^{-}(y)+u^{+}(y) u^{-}(x)\right) \\
& \geq 0 \nonumber.
\end{align}
Therefore, we get that
\begin{align*}
 |\nabla^s u|^{2}&=\frac{1}{2} \sum_{y\in\mathbb{Z}^d, y \neq x} W_{s}(x, y)\left(u(x)-u(y)\right)^2 \\
&=\frac{1}{2} \sum_{y\in\mathbb{Z}^d, y \neq x} W_{s}(x, y)\left(u^+(x)-u^+(y)+u^-(x)-u^-(y)\right)^2 \\
&=\frac{1}{2} \sum_{y\in\mathbb{Z}^d, y \neq x} W_{s}(x, y)\left[(u^+(x)-u^+(y))^2+(u^-(x)-u^-(y))^2+2(u^+(x)-u^+(y))(u^-(x)-u^-(y))\right] \\
&= \left|\nabla^s u^{+}\right|^{2}+\left|\nabla^s u^{-}\right|^{2}+2 \nabla^s u^{+} \nabla ^su^{-} \\ &\geq\left|\nabla^s u^{-}\right|^{2}.   
\end{align*}
This means that for any $2\leq p<\infty$,
\begin{equation*}
\left|\nabla^s u^{-}\right|^{p} \leq|\nabla^s u|^{p} . 
\end{equation*}
Moreover, we have $\left|u^{-}\right|^{p} \leq|u|^{p}$. Then for $u \in H_{s,p}$,  we obtain that $u^{-} \in H_{s, p}$.
A similar argument says $u^+\in H_{s,p}$. Since $H_{s,p}$ is a Banach space and $|u|=u^+-u^-$, we get that $|u|\in H_{s,p}.$
\end{proof}

Now we introduce the well-known discrete Hardy-Littlewood-Sobolev (HLS for abbreviation) inequality, see \cite{LW1,W1}.

\begin{lm}\label{lm1}

Let $0<\alpha <d,\,1<r,t<\infty$ and $\frac{1}{r}+\frac{1}{t}+\frac{d-\alpha}{d}=2$. Then we have the discrete
HLS inequality
\begin{equation}\label{bo}
\int_{\mathbb{Z}^d}(R_\alpha\ast u)(x)v(x)\,d\mu\leq C_{r,t,\alpha,d}\|u\|_r\|v\|_t,\quad u\in \ell^r(\mathbb{Z}^d),\,v\in \ell^t(\mathbb{Z}^d).
\end{equation}
And an equivalent form is
\begin{eqnarray*}\label{p1}
\|R_\alpha\ast u\|_{\frac{dr}{d-\alpha r}}\leq C_{r,\alpha,d}\|u\|_r,\quad u\in \ell^r(\mathbb{Z}^d),
\end{eqnarray*}
where $1<r<\frac{d}{\alpha}$.
\end{lm}

\begin{lm}\label{lhf}
Let $(f_1)$-$(f_3)$ hold. If $u\in H_{s,p}\backslash\{0\}$ with $u\geq 0$, then for $t\geq 1$, we have
\begin{equation}\label{95}
 \int_{\mathbb{Z}^d}(R_\alpha \ast F(tu))F(tu)\,d\mu\geq t^\theta\int_{\mathbb{Z}^d}(R_\alpha \ast F(u))F(u)\,d\mu.   
\end{equation}
\end{lm}
\begin{proof}
The proof is similar to that of \cite[Lemma 2.5]{W8}. We omit it here.
\end{proof}

In the following, we introduce a compactness result for $H_{s,p}$, which  can be seen in \cite{ZY}.

\begin{lm}\label{lgg}
Let $(h_1)$, $(h_2)$ hold. Then for any $q\in [p,\infty)$, $H_{s,p}$ is embedded compactly into $\ell^{q}(\mathbb{Z}^d)$. That is, for any bounded sequence $\left\{u_{n}\right\} \subset H_{s,p}$, there exists $u \in H_{s,p}$ such that, up to a subsequence, $u_{n} \rightharpoonup u$ weakly in $H_{s,p}$ and $u_{n} \rightarrow u$ strongly in $ \ell^{q}(\mathbb{Z}^d).$
\end{lm}

\begin{lm}\label{lo}
Let $(h_1)$, $(h_2)$ and  $(f_1)$-$(f_3)$ hold. If $\left\{u_{n}\right\} $ is bounded in $H_{s,p}$ and $u_n\geq 0$, then there exists $u \in H_{s,p}$ with $u\geq 0$ such that
$$\lim\limits_{n\rightarrow\infty} \int_{\mathbb{Z}^d}(R_\alpha \ast F(u_n))F(u_n)\,d\mu=\int_{\mathbb{Z}^d}(R_\alpha \ast F(u))F(u)\,d\mu,$$
and
$$\lim\limits_{n\rightarrow\infty}\int_{\mathbb{Z}^d}(R_\alpha \ast F(u_n))f(u_n)u_n\,d\mu=\int_{\mathbb{Z}^d}(R_\alpha \ast F(u))f(u)u\,d\mu.$$

\end{lm}
\begin{proof}
Note that $\left\{u_{n}\right\} $ is bounded in $H_{s,p}$.
On the one hand, by Lemma \ref{lgg}, there exists $u \in H_{s,p}$ such that
\begin{equation}\label{vv}
 u_n\rightharpoonup u,\quad \text{in}~H_{s,p},\qquad u_n\rightarrow u,\quad \text{strongly}~\text{in}~\ell^q(\mathbb{Z}^d), ~q\geq p.   
\end{equation}
On the other hand, by (\ref{lx}), we have that $\{u_n\}$ is bounded in $\ell^p(\mathbb{Z}^d)$, and hence bounded in $\ell^\infty(\mathbb{Z}^d)$. Therefore, by diagonal principle, there exists a subsequence of $\{u_n\}$  pointwise converging to $u$, that is,
\begin{equation}\label{an}
u_n\rightarrow u,\quad \text{pointwise}~\text{in}~\mathbb{Z}^d.   \end{equation}
Since $u_n\geq 0$, we have $u\geq 0$.

(i) A direct calculation yields that
$$
\begin{aligned}
& \left|\int_{\mathbb{Z}^d}(R_\alpha \ast F(u_n))F(u_n)\,d\mu-\int_{\mathbb{Z}^d}(R_\alpha \ast F(u))F(u)\,d\mu\right|\\ \leq &\left|\int_{\mathbb{Z}^d}(R_\alpha \ast F(u_n))\left(F(u_n)-F(u)\right)\,d\mu\right|+\left|\int_{\mathbb{Z}^d}\left(R_\alpha \ast (F(u_n)-F(u)\right)F(u)\,d\mu\right| \\ =&: I_1+I_2.  
\end{aligned}
$$
First, we claim that
\begin{equation}\label{ah}
\lim\limits_{n\rightarrow\infty}\int_{\mathbb{Z}^d} |F(u_n)-F(u)|^{\frac{2d}{d+\alpha}}\,d\mu=0.  \end{equation}
In fact, by (\ref{an}), one gets that $F(u_{n}) \rightarrow F(u)$ pointwise in $\mathbb{Z}^d$. By (\ref{ae}), we deduce that
$$
\begin{aligned}
 |F(u_n)-F(u)|^{\frac{2d}{d+\alpha}}\leq &C\left( |F(u_n)|^{\frac{2d}{d+\alpha}} +|F(u)|^{\frac{2d}{d+\alpha}} \right)\\ \leq &C\left(u_n^p+u_n^\tau\right)^{\frac{2d}{d+\alpha}}+C |F(u)|^{\frac{2d}{d+\alpha}} \\ \leq & C\left(u_n^{\frac{2dp}{d+\alpha}}+u_n^{\frac{2d\tau}{d+\alpha}}\right)+C |F(u)|^{\frac{2d}{d+\alpha}}.
\end{aligned}
$$
Notice that $\frac{2dp}{d+\alpha}>p$ and $\frac{2d\tau}{d+\alpha}>p$. By (\ref{vv}), (\ref{an}) and Lebesgue dominated convergence theorem, we get the claim (\ref{ah}).

For $I_1$,  it follows from (\ref{ae}), the HLS inequality (\ref{bo}) and (\ref{ah}) that
$$
\begin{aligned}
    I_1=&\left|\int_{\mathbb{Z}^d}(R_\alpha \ast F(u_n))(F(u_n)-F(u))\,d\mu\right|\\ \leq & C\left(\int_{\mathbb{Z}^d} |F(u_n)|^{\frac{2d}{d+\alpha}}\,d\mu  \right)^{\frac{d+\alpha}{2d}}\left(\int_{\mathbb{Z}^d} |F(u_n)-F(u)|^{\frac{2d}{d+\alpha}}\,d\mu  \right)^{\frac{d+\alpha}{2d}}\\ \leq & C\left(\int_{\mathbb{Z}^{d}}\left(u_{n}^{p}+u_{n}^{\tau}\right)^{\frac{2d}{d+\alpha}} d \mu\right)^{\frac{d+\alpha}{2 d}}\left(\int_{\mathbb{Z}^d} |F(u_n)-F(u_0)|^{\frac{2d}{d+\alpha}}\,d\mu  \right)^{\frac{d+\alpha}{2d}}\\ \leq & C
\left(\left\|u_{n}\right\|_{\frac{2dp }{d+\alpha}}^{p}+\left\|u_{n}\right\|_{\frac{2 d\tau}{d+\alpha}}^{\tau}\right)\left(\int_{\mathbb{Z}^d} |F(u_n)-F(u_0)|^{\frac{2d}{d+\alpha}}\,d\mu  \right)^{\frac{d+\alpha}{2d}}\\ \leq &C
\left(\left\|u_{n}\right\|_{H_{s,p}}^{p}+\left\|u_{n}\right\|_{H_{s,p}}^{\tau}\right)\left(\int_{\mathbb{Z}^d} |F(u_n)-F(u_0)|^{\frac{2d}{d+\alpha}}\,d\mu  \right)^{\frac{d+\alpha}{2d}}\\ \leq & C\left(\int_{\mathbb{Z}^d} |F(u_n)-F(u_0)|^{\frac{2d}{d+\alpha}}\,d\mu  \right)^{\frac{d+\alpha}{2d}}\\ \rightarrow &0, \quad n\rightarrow\infty,
\end{aligned}
$$ 
where we have used  $\|u_n\|_{\frac{2dp}{d+\alpha}}\leq \|u_n\|_{p}\leq C\|u_n\|_{H_{s,p}}$ since $\frac{2dp}{d+\alpha}>p$ in the forth inequality.

For $I_2$, similarly, we have
$$
\begin{aligned}
    I_2=&\left|\int_{\mathbb{Z}^d}\left(R_\alpha \ast (F(u_n)-F(u)\right)F(u)\,d\mu\right|\\ \leq & C\left(\int_{\mathbb{Z}^d} |F(u_n)-F(u)|^{\frac{2d}{d+\alpha}}\,d\mu  \right)^{\frac{d+\alpha}{2d}}\left(\int_{\mathbb{Z}^d} |F(u)|^{\frac{2d}{d+\alpha}}\,d\mu  \right)^{\frac{d+\alpha}{2d}}\\ \leq & C\left(\int_{\mathbb{Z}^d} |F(u_n)-F(u_0)|^{\frac{2d}{d+\alpha}}\,d\mu  \right)^{\frac{d+\alpha}{2d}}\\ \rightarrow &0,\quad n\rightarrow\infty.
\end{aligned}
$$
 The above arguments imply $$\lim\limits_{n\rightarrow\infty} \int_{\mathbb{Z}^d}(R_\alpha \ast F(u_n))F(u_n)\,d\mu=\int_{\mathbb{Z}^d}(R_\alpha \ast F(u))F(u)\,d\mu.$$

 \
 \

 (ii) A straightforward calculation gives that $$
\begin{aligned}
&\left|\int_{\mathbb{Z}^d}(R_\alpha \ast F(u_n))f(u_n)u_n\,d\mu-\int_{\mathbb{Z}^d}(R_\alpha \ast F(u_0))f(u_0)u_0\,d\mu\right|\\ =&\left|\int_{\mathbb{Z}^d}(R_\alpha \ast F(u_n))f(u_n)u_0\,d\mu-\int_{\mathbb{Z}^d}(R_\alpha \ast F(u_0))f(u_0)u_0\,d\mu+\int_{\mathbb{Z}^d}(R_\alpha \ast F(u_n))f(u_n)(u_n-u_0)\,d\mu\right|\\\leq &\left|\int_{\mathbb{Z}^d}(R_\alpha \ast \left(F(u_n)-F(u_0)\right)f(u_0)u_0\,d\mu\right|+\left|\int_{\mathbb{Z}^d}(R_\alpha \ast F(u_n))\left(f(u_n)-f(u_0)\right)u_0\,d\mu\right|\\&+\left|\int_{\mathbb{Z}^d}(R_\alpha \ast F(u_n))f(u_n)(u_n-u_0)\,d\mu\right|.
\end{aligned}
$$
We prove that
\begin{equation}\label{ag}
\lim\limits_{n\rightarrow\infty} \int_{\mathbb{Z}^d}\left|(f(u_n)-f(u_0))\right|^{\frac{2d}{d+\alpha}\frac{\tau p}{(\tau-1)(p-1)}}\,d\mu=0.  \end{equation}
Indeed, by (\ref{an}), one gets that $f(u_{n}) \rightarrow f(u_0)$ pointwise in $\mathbb{Z}^d$. By (\ref{ad}), we deduce that
$$
\begin{aligned}
 |f(u_n)-f(u_0)|^{\frac{2d}{d+\alpha}\frac{\tau p}{(\tau-1)(p-1)}}\leq &C\left( |f(u_n)|^{\frac{2d}{d+\alpha}\frac{\tau p}{(\tau-1)(p-1)}} +|f(u_0)|^{\frac{2d}{d+\alpha}\frac{\tau p}{(\tau-1)(p-1)}} \right)\\ \leq &C\left(u_n^{p-1}+u_n^{\tau-1}\right)^{\frac{2d}{d+\alpha}\frac{\tau p}{(\tau-1)(p-1)}}+C |f(u_0)|^{\frac{2d}{d+\alpha}\frac{\tau p}{(\tau-1)(p-1)}} \\ \leq & C\left(u_n^{\frac{2d}{d+\alpha}\frac{\tau p}{(\tau-1)}}+u_n^{\frac{2d}{d+\alpha}\frac{\tau p}{(p-1)}}\right)+C |f(u_0)|^{\frac{2d}{d+\alpha}\frac{\tau p}{(\tau-1)(p-1)}}.
\end{aligned}
$$
Note that $\frac{2d}{d+\alpha}\frac{\tau p}{(p-1)}>p$ and $\frac{2d}{d+\alpha}\frac{\tau p}{(\tau-1)}>p$. By (\ref{vv}), (\ref{an}) and Lebesgue dominated convergence theorem, we get the claim (\ref{ag}).
The rest proof is similar to that of (i), one can get $$\lim\limits_{n\rightarrow\infty}\int_{\mathbb{Z}^d}(R_\alpha \ast F(u_n))f(u_n)u_n\,d\mu=\int_{\mathbb{Z}^d}(R_\alpha \ast F(u))f(u)u\,d\mu.$$ 

\end{proof}

Finally, in order to prove our results, we make a simple reduction argument and assume $f(t) \equiv 0$ for all $t<0$. Let
$$
\widetilde{f}(t)= \begin{cases}0, & t<0, \\ f(t), & t \geq 0.\end{cases}
$$
By ($f_3$), we get that $\widetilde{f} \geq 0$. Let $u\in H_{s,p}$ be a nontrivial weak solution to the equation
\begin{equation}\label{aq}
(-\Delta)^s_{p} u+h(x)|u|^{p-2} u=\left(R_{\alpha} *\widetilde{F}(u)\right)\widetilde{f}(u) 
\end{equation}
on lattice graphs $\mathbb{Z}^{d} $. Then we have the following result.

\begin{lm}\label{i0}
If $u \in H_{s,p}$ is a nontrivial weak solution to the equation  (\ref{aq}), then it is a strictly positive solution to the equation  (\ref{aa}).    
\end{lm}

\begin{proof}
Let $u \in H_{s,p}$ be a nontrivial weak solution to the equation (\ref{aq}). 
By Lemma \ref{i1}, one gets that $u^-\in H_{s,p}$. Moreover, we have that
\begin{align*}
\nabla^s u \nabla^s u^{-}&=\frac{1}{2} \sum_{y\in\mathbb{Z}^d, y \neq x} W_{s}(x, y)\left(u(x)-u(y)\right)\left(u^-(x)-u^-(y)\right) \\
&=\frac{1}{2} \sum_{y\in\mathbb{Z}^d, y \neq x} W_{s}(x, y)\left[(u^+(x)-u^+(y))+(u^-(x)-u^-(y))\right]\left(u^-(x)-u^-(y)\right)\\&=
\nabla^s u^{+} \nabla^s u^{-}+\left|\nabla^s u^{-}\right|^{2} \\&\geq\left|\nabla^s u^{-}\right|^{2}, 
\end{align*}
where we have used  (\ref{ar}) in the last inequality.
Now we take $u^{-}$as the test function of the equation (\ref{aq}), then one gets that
\begin{align*}
0 & \geq \int_{\mathbb{Z}^d}(R_\alpha \ast \widetilde{F}(u))\widetilde{f}(u)u^-\,d\mu \\
& =\int_{\mathbb{Z}^d}\left(|\nabla^s u|^{p-2} \nabla^s u \nabla^s u^{-}+h(x)|u|^{p-2} u u^{-}\right) d \mu \\
& \geq \int_{\mathbb{Z}^d}\left(\left|\nabla^s u^{-}\right|^{p}+h(x)\left|u^{-}\right|^{p}\right) d \mu \\
& =\left\|u^{-}\right\|_{H_{s,p}} . 
\end{align*}
This means that $u^{-} \equiv 0$, and hence $u \geq 0$.  If there exists some $x_{0}\in\mathbb{Z}^d$ such that $u\left(x_{0}\right)=0$, then we have $(-\Delta)^s_{p} u\left(x_{0}\right)=0$, namely
$$
\sum_{y\in\mathbb{Z}^d, y\neq x_{0}} W_{s}(x_0,y)\left(|\nabla^s u|^{p-2}(x_0)+|\nabla^s u|^{p-2}(y)\right) u(y)=0 .
$$
Consequently, $u(y)=0$ for all $y\in\mathbb{Z}^d$. However, this is a contradiction since $u$ is nontrivial. Thus we get that $u(x)>0$ for all $x \in \mathbb{Z}^d$, and thus $\widetilde{f}(u)=f(u)$. Therefore, $u$ is a strictly positive solution to the equation (\ref{aa}).
\end{proof}

As a consequence, without loss of generality, we assume $f(t)=0$ for $t \leq 0$ in our proofs. Clearly, the equation (\ref{aq}) is equivalent to the equation
\begin{equation}\label{qq}
(-\Delta)^s_{p} u+h(x)|u|^{p-2} u=\left(R_{\alpha} *F(u^+)\right)f(u^+),\quad u\in H_{s,p}. 
\end{equation}
Therefore, in the sequel, we only need to prove the equation (\ref{qq}) has a nontrivial weak solution, and hence a strictly positive solution of the equation (\ref{aa}). Let $\widetilde{J}_{s,p}: H_{s,p}\rightarrow \mathbb{R}$ be the functional with respect to the equation (\ref{qq})
$$\widetilde{J}_{s,p}(u)=\frac{1}{p}\|u\|^p_{H_{s,p}}-\int_{\mathbb{Z}^d} \left(R_{\alpha} *F(u^+)\right)F(u^+)\,d\mu.$$
Moreover, for any $\phi\in H_{s,p}$, we have
\begin{eqnarray*}
\left\langle \widetilde{J}_{s,p}^{\prime}(u), \phi\right\rangle=\int_{\mathbb{Z}^d}\left(|\nabla^s u|^{p-2}\nabla^s u \nabla^s \phi+h(x)|u|^{p-2} u \phi\right) \,d\mu-\int_{\mathbb{Z}^d}(R_\alpha \ast F(u^+))f(u^+)\phi\,d\mu.
\end{eqnarray*}

\
\

\section{Proof of Theorem \ref{t2}}
In this section, we prove Theorem \ref{t2} by the mountain pass theorem.  
Recall that, for a given functional $I\in C^{1}(X,\mathbb{R})$, a sequence $\{u_n\}\subset X$ is a Palais-Smale sequence at level $c\in\mathbb{R}$, $(PS)_c$ sequence for short, of the functional $I$, if it satisfies, as $n\rightarrow\infty$,
\begin{eqnarray*}
I(u_n)\rightarrow c, \qquad \text{in}~ X,\qquad\text{and}\qquad
I'(u_n)\rightarrow 0, \qquad \text{in}~X^*,
\end{eqnarray*}
where $X$ is a Banach space and $X^{*}$ is the dual space of $X$. Moreover, we say that $I$ satisfies $(PS)_c$ condition, if any $(PS)_c$ sequence has a convergent subsequence. 

First we show that
 the functional $\widetilde{J}_{s,p}$ satisfies the mountain-pass geometry.
\begin{lm}\label{lm}
 Let $(h_1)$ and $(f_1)$-$(f_3)$ hold. Then
\begin{itemize}
    \item[(i)] there exist $\sigma, \rho>0$ such that $\widetilde{J}_{s,p}(u) \geq \sigma>0$ for $\|u\|_{H_{s,p}}=\rho$;
    \item[(ii)] there exists $0< e \in H_{s,p}$ with $\|e\|_{H_{s,p}}>\rho$ such that $\widetilde{J}_{s,p}(e)< 0$.  
\end{itemize}  
\end{lm}
\begin{proof}
  (i) Let $u\in H_{s,p}\backslash\{0\}$.
Since $\alpha\in(0,d)$ and $\tau>\frac{(d+\alpha)p}{2d}$,
by (\ref{ae}), the HLS inequality (\ref{bo}) and $u^+\leq |u|$, we deduce that
\begin{eqnarray}\label{mc}
\int_{\mathbb{Z}^d}\left(R_\alpha \ast F(u^+)\right)F(u^+)\,d\mu\nonumber&\leq& C\left(\int_{\mathbb{Z}^d} |F(u^+)|^{\frac{2d}{d+\alpha}}\,d\mu  \right)^{\frac{d+\alpha}{d}}\\&\leq &C\left(\int_{\mathbb{Z}^d} \left(\varepsilon |u|^p+C_\varepsilon |u|^{\tau}\right)^{\frac{2d}{d+\alpha}}\,d\mu  \right)^{\frac{d+\alpha}{d}}\nonumber\\&\leq& \varepsilon \|u\|_{\frac{2dp}{d+\alpha}}^{2p}+C_\varepsilon \|u\|_{\frac{2d\tau}{d+\alpha}}^{2\tau}\nonumber\\&\leq& \varepsilon \|u\|_{p}^{2p}+C_\varepsilon \|u\|_{p}^{2\tau}\nonumber\\&\leq& \varepsilon \|u\|_{H_{s,p}}^{2p}+C_\varepsilon \|u\|_{H_{s,p}}^{2\tau},
\end{eqnarray}
where we have used the fact $\|u\|_{\frac{2dp}{d+\alpha}}\leq \|u\|_{p}$ since $\frac{2dp}{d+\alpha}>p$ in the forth inequality.
Then we have
\begin{equation}\label{jv}
\begin{aligned}
\widetilde{J}_{s,p}(u) & =\frac{1}{p}\|u\|_{H_{s,p}}^{p}-\frac{1}{2}\int_{\mathbb{Z}^d}(R_\alpha \ast F(u^+))F(u^+)\,d\mu \\&\geq\frac{1}{p}\|u\|_{H_{s,p}}^{p}-\varepsilon \|u\|_{H_{s,p}}^{2p}-C_\varepsilon \|u\|_{H_{s,p}}^{2\tau}.
\end{aligned}
\end{equation}
Note that $2\tau>\frac{(d+\alpha)p}{d}>p$. Let $\varepsilon\rightarrow 0^+$, then there exist $\sigma, \rho>0$ small enough such that $\widetilde{J}_{s,p}(u) \geq \sigma>0$ for $\|u\|_{H_{s,p}}=\rho$.

(ii)  Let $u\in H_{s,p}\backslash\{0\}$ with $u\geq 0$ be fixed. Since $\theta>p$ and $u^+\in H_{s,p}$, by (\ref{95}), we obtain that
\begin{eqnarray}\label{md}
 \lim _{t\rightarrow\infty} \widetilde{J}_{s,p}(tu)\nonumber&=&\lim _{t \rightarrow\infty}\left[\frac{t^{p}}{p}\|u\|_{H_{s,p}}^p-\frac{1}{2} \int_{\mathbb{Z}^d}(R_\alpha \ast F(tu^+))F(tu^+)\,d\mu\right]\nonumber\\&\leq&\lim _{t\rightarrow\infty}\left[\frac{t^{p}}{p}\|u\|_{H_{s,p}}^p-\frac{t^\theta}{2} \int_{\mathbb{Z}^d}\left(R_\alpha \ast F(u^+)\right)F(u^+)\,d\mu\right]\nonumber\\&\rightarrow&-\infty, \quad t \rightarrow \infty.  \end{eqnarray}
Hence, there exists $t_{0}>0$ large enough such that $\left\|t_{0} u\right\|>\rho$ and $J\left(t_{0} u\right)<0$. The proof is completed by taking $e=t_{0}u >0$. 
\end{proof}

Next, we show that the functional $\widetilde{J}_{s,p}$ satisfies the $(PS)_c$ condition.

\begin{lm}\label{ln}
Let $(h_1),(h_2)$ and $(f_1)$-$(f_3)$ hold. Then for any $c\in\mathbb{R}$, $\widetilde{J}_{s,p}$ satisfies the $(PS)_c$ condition. 
\end{lm}
\begin{proof}
For any $c\in\mathbb{R}$, let $\left\{u_{n}\right\}$ be a $(P S)_{c}$ sequence of $\widetilde{J}_{s,p}$,
\begin{equation*}
 \widetilde{J}_{s,p}\left(u_{n}\right)\rightarrow c, \quad \text { and } \quad \widetilde{J}_{s,p}^{\prime}\left(u_{n}\right)\rightarrow0.
 \end{equation*}
Then it follows from $(f_3)$ that
\begin{eqnarray}\label{ba}
\|u_n\|_{H_{s,p}}^p\nonumber&=&\frac{p}{2}\int_{\mathbb{Z}^d}(R_\alpha \ast F(u^+_n))F(u^+_n)\,d\mu+pc+o_n(1)\nonumber\\&\leq& \frac{p}{\theta}\int_{\mathbb{Z}^d}(R_\alpha \ast F(u^+_n))f(u^+_n)u^+_n\,d\mu+pc+o_n(1)\nonumber\\&\leq&\frac{p}{\theta} \left(\|u_n\|_{H_{s,p}}^p+o_n(1)\|u_n\|_{s,p}\right)+pc+o_n(1)\nonumber\\&=&\frac{p}{\theta}\|u_n\|_{H_{s,p}}^p+o_n(1)\|u_n\|_{H_{s,p}}+pc+o_n(1),
\end{eqnarray}
where $o_n(1)\rightarrow 0$ as $n\rightarrow\infty.$
This inequality implies that $\{u_n\}$ is bounded in $H_{s,p}$. Then by Lemma \ref{lgg}, passing to a subsequence if necessary, there exists $u \in H_{s,p}$ such that
\begin{equation}\label{ua}
\begin{cases}u_{n} \rightharpoonup u, & \text { in } H_{s,p}, \\ u_{n} \rightarrow u, & \text { in } \ell^{q}(\mathbb{Z}^d),q\geq p.\end{cases}
\end{equation}
By the HLS inequality (\ref{bo}), H\"{o}lder inequality, the boundedness of $\{u_n\}$ and (\ref{ua}), we deduce that
\begin{equation}\label{ak}
\begin{aligned}
&\left|\int_{\mathbb{Z}^d}(R_\alpha \ast F(u^+_n))f(u^+_n)(u_n-u)\,d\mu\right|
\\ \leq& C\left(\int_{\mathbb{Z}^d} |F(u^+_n)|^{\frac{2d}{d+\alpha}}\,d\mu  \right)^{\frac{d+\alpha}{2d}}\left(\int_{\mathbb{Z}^d} |f(u^+_n)(u_n-u)|^{\frac{2d}{d+\alpha}}\,d\mu  \right)^{\frac{d+\alpha}{2d}}\\ \leq&C\left(\int_{\mathbb{Z}^d} (|u_n|^p+|u_n|^\tau)^{\frac{2d}{d+\alpha}}\,d\mu  \right)^{\frac{d+\alpha}{2d}}\left(\int_{\mathbb{Z}^d} \left[(|u_n|^{p-1}+|u_n|^{\tau-1})|u_n-u|\right]^{\frac{2d}{d+\alpha}}\,d\mu  \right)^{\frac{d+\alpha}{2d}}\\ \leq& C\left[\left(\int_{\mathbb{Z}^d} (|u_n|^{p-1}|u_n-u|)^{\frac{2d}{d+\alpha}}\,d\mu  \right)^{\frac{d+\alpha}{2d}}+\left(\int_{\mathbb{Z}^d} (|u_n|^{\tau-1}|u_n-u|)^{\frac{2d}{d+\alpha}}\,d\mu  \right)^{\frac{d+\alpha}{2d}}\right] \\ \leq& C\|u_n\|^{p-1}_{\frac{2dp}{d+\alpha}}\|u_n-u\|_{\frac{2dp}{d+\alpha}}+C\|u_n\|^{\tau-1}_{\frac{2d\tau}{d+\alpha}}\|u_n-u\|_{\frac{2d\tau}{d+\alpha}} \\ \leq& C\|u_n-u\|_{\frac{2dp}{d+\alpha}}+C\|u_n-u\|_{\frac{2d\tau}{d+\alpha}} \\ \rightarrow&0,\quad n\rightarrow\infty,
\end{aligned}
\end{equation}
where we have used the facts $\frac{2dp}{d+\alpha}>p$ and $\frac{2d\tau}{d+\alpha}>p$ in the last inequality. 

Since $\{u_n\}$ is bounded in $H_{s,p}$ and $\widetilde{J}'_{s,p}(u_n)\rightarrow 0$, for  $\phi\in H_{s,p}$, we have that
$\langle \widetilde{J}'_{s,p}(u_n),\phi\rangle=o_n(1)$,  which is equivalent to

\begin{equation*}
\int_{\mathbb{Z}^d}\left(\left|\nabla^{s} u_{n}\right|^{p-2} \nabla^{s} u_{n} \nabla^{s} \phi+h(x)\left|u_{n}\right|^{p-2} u_{n} \phi\right) d \mu=\int_{\mathbb{Z}^d}(R_\alpha \ast F(u^+_n))f(u^+_n)\phi\,d\mu+o_{n}(1). 
\end{equation*}
Let $\phi=(u_{n}-u)\in H_{s,p}$ in the above equality. Then by (\ref{ak}), we derive that
\begin{equation*}
\int_{\mathbb{Z}^d}\left(\left|\nabla^{s} u_{n}\right|^{p-2} \nabla^{s} u_{n} \nabla^{s}\left(u_{n}-u\right)+h(x)\left|u_{n}\right|^{p-2} u_{n}\left(u_{n}-u\right)\right) d \mu=o_{n}(1).
\end{equation*}
Moreover, by the fact that $u_{n}$ weakly converges to $u$ in $H_{s, p}$, we derive that
\begin{equation*}
\int_{\mathbb{Z}^d}\left(\left|\nabla^{s} u\right|^{p-2} \nabla^{s} u \nabla^{s}\left(u_{n}-u\right)+h(x)|u|^{p-2} u\left(u_{n}-u\right)\right) d \mu=o_{n}(1).
\end{equation*}
By Lemma 15 in \cite{ZY}, one gets that
$$
\begin{aligned}
\left\|u_{n}-u\right\|_{H_{s, p}}^{p}= & \int_{\mathbb{Z}^d}\left(\left|\nabla^{s}\left(u_{n}-u\right)\right|^{p}+h(x)\left|u_{n}-u\right|^{p}\right) d \mu \\
\leq & 2^{p-2} p \int_{\mathbb{Z}^d}\left(\left|\nabla^{s} u_{n}\right|^{p-2} \nabla^{s} u_{n} \nabla^{s}\left(u_{n}-u\right)+h(x)\left|u_{n}\right|^{p-2} u_{n}\left(u_{n}-u\right)\right) d \mu \\
& +2^{p-2} p \int_{\mathbb{Z}^d}\left(\left|\nabla^{s} u\right|^{p-2} \nabla^{s} u \nabla^{s}\left(u_{n}-u\right)+h(x)|u|^{p-2} u\left(u_{n}-u\right)\right) d \mu
\\ \rightarrow &0,\quad n\rightarrow\infty.
\end{aligned}
$$
This means that $\left\|u_{n}-u\right\|_{H_{s, p}}=o_{n}(1)$.

\end{proof}

\
\

{\bf Proof of Theorem \ref{t2}. } By Lemma \ref{lm} and Lemma \ref{ln}, one sees that $\widetilde{J}_{s,p}$ satisfies the geometric conditions of the mountain-pass theorem. More precisely, Lemma \ref{lm} guarantees that there exists a positive function $u_{0} \in H_{s, p}$ with $\|u_{0}\|_{H_{s, p}}>\rho>0$ such that
$$
\inf _{\|u\|_{H_{s, p}}=\rho} \widetilde{J}_{s, p}(u)>\widetilde{J}_{s, p}(0)=0>\widetilde{J}_{s, p}(u_{0})
$$
for some constant $\rho>0$.  Lemma \ref{ln} ensures that $\widetilde{J}_{s,p}$ satisfies the $(P S)_{c}$ condition for any $c \in \mathbb{R}$. Define a set of paths
$$
\Gamma=\left\{\gamma \in C\left([0,1], H_{s, p}\right): \gamma(0)=0, \gamma(1)=u_{0}\right\} .
$$
By the mountain-pass theorem, we obtain that
$$
c=\inf_{\gamma \in \Gamma} \max _{t \in[0,1]} \widetilde{J}_{s, p}(\gamma(t))
$$
is a critical value of $\widetilde{J}_{s, p}$. In particular, there exists $u \in H_{s, p}$ such that $\widetilde{J}_{s, p}(u)=c$. Since $\widetilde{J}_{s, p}(u)=c\geq \rho>0$, we have $u \neq0$. Hence, $u$ is a nontrivial weak solution to the equation (\ref{qq}). Then it follows from Lemma \ref{i0} that $u$ is a strictly positive solution to the equation (\ref{aa}). The proof is completed.\qed

\
\

\section{Proof of Theorem \ref{t1}}
In this section, we prove Theorem \ref{t1} by the method of Nehari manifold. Let $$\mathcal{M}_{s, p}=\left\{u \in H_{s, p} \backslash\{0\}:\left\langle \widetilde{J}_{s, p}^{\prime}(u), u\right\rangle=0\right\}$$ be the Nehari manifold corresponding to $\widetilde{J}_{s,p}$, namely
\begin{equation*}
\mathcal{M}_{s, p}=\left\{u \in H_{s, p} \backslash\{0\}: \int_{\mathbb{Z}^d}\left(\left|\nabla^{s} u\right|^{p}+h(x)|u|^{p}\right) d \mu=\int_{\mathbb{Z}^d}(R_\alpha \ast F(u^+))f(u^+)u^+\,d\mu\right\} . 
\end{equation*}
Moreover, we denote
$$c_{s,p}:=\inf\limits_{u\in \mathcal{M}_{s,p}} \widetilde{J}_{s,p}(u).$$

First, we show some properties of $\widetilde{J}_{s,p}$ on the Nehari manifold $\mathcal{M}_{s,p}$ that are useful in our proofs.

\begin{lm}\label{lu}
For any $u \in \mathcal{M}_{s,p}$, there exists $\eta>0$ such that $\|u\|_{H_{s,p}} \geq \eta>0$.  \end{lm}

\begin{proof}
 It follows from the HLS inequalty (\ref{bo}) that
\begin{eqnarray}\label{aj}
\int_{\mathbb{Z}^d}(R_\alpha \ast F(u^+))f(u^+)u^+\nonumber\,d\mu
&\leq& C\left(\int_{\mathbb{Z}^d} |F(u^+)|^{\frac{2d}{d+\alpha}}\,d\mu  \right)^{\frac{d+\alpha}{2d}}\left(\int_{\mathbb{Z}^d} |f(u^+)u^+|^{\frac{2d}{d+\alpha}}\,d\mu  \right)^{\frac{d+\alpha}{2d}}\nonumber\\&\leq&C\left(\int_{\mathbb{Z}^d} \left(\varepsilon|u|^p+C_\varepsilon|u|^{\tau}\right)^{\frac{2d}{d+\alpha}}\,d\mu  \right)^{\frac{d+\alpha}{d}}\nonumber\\&\leq& \varepsilon\|u\|^{2p}_{\frac{2dp}{d+\alpha}}+C_\varepsilon\|u\|^{2\tau}_{\frac{2d\tau}{d+\alpha}}\nonumber\\&\leq& \varepsilon\|u\|_{H_{s,p}}^{2p}+C_\varepsilon\|u\|_{H_{s,p}}^{2\tau}.
\end{eqnarray}
Then for $u \in\mathcal{M}_{s,p}$, we have that
$$
\begin{aligned}
0&=\left\langle \widetilde{J}^{\prime}(u), u\right\rangle\\ & =\|u\|_{H_{s,p}}^{p}-\int_{\mathbb{Z}^d}(R_\alpha \ast F(u^+))f(u^+)u^+\,d\mu  \\
& \geq\|u\|_{H_{s,p}}^{p}-\varepsilon\|u\|_{H_{s,p}}^{2p}-C_{\varepsilon}\|u\|_{H_{s,p}}^{2\tau} .
\end{aligned}
$$
Note that $2\tau>p$, one gets easily that there exists $\eta>0$ such that $\|u\|_{H_{s,p}} \geq \eta>0$.   
\end{proof}

\begin{lm}\label{ld}
 For $u\in\mathcal{M}_{s,p}$, we have $\widetilde{J}_{s,p}(u)>0$.   
\end{lm}

\begin{proof}
    For any $u \in \mathcal{M}_{s,p}$, by $(f_3)$ and Lemma \ref{lu}, one gets that
\begin{eqnarray*}
\widetilde{J}_{s,p}(u) &=& \widetilde{J}_{s,p}(u)-\frac{1}{\theta}\left\langle \widetilde{J}_{s,p}^{\prime}(u), u\right\rangle \\
&=&\left(\frac{1}{p}-\frac{1}{\theta}\right)\|u\|_{H_{s,p}}^{p}+\frac{1}{2}\int_{\mathbb{Z}^d}(R_\alpha \ast F(u^+))\left(\frac{2}{\theta}f(u^+)u^+-F(u^+)\right)\,d\mu\\
& \geq &\left(\frac{1}{p}-\frac{1}{\theta}\right)\|u\|_{H_{s,p}}^{p} \\& \geq &\left(\frac{1}{p}-\frac{1}{\theta}\right) \eta^{p}\\ &>& 0.
\end{eqnarray*}
\end{proof}

\begin{lm}\label{lj}
 Let $\left(h_{1}\right)$ and $\left(f_{1}\right)$-$\left(f_{4}\right)$ hold. Then for any $u \in H_{s,p} \backslash\{0\}$, there exists a unique $t_{u}>0$ such that $t_{u} u \in \mathcal{M}_{s,p}$ and $\widetilde{J}_{s,p}(t_{u} u)=$ $\max\limits_{t>0} \widetilde{J}_{s,p}(tu)$. Moreover, if $u \in \mathcal{M}_{s, p}$, then $\widetilde{J}_{s, p}(u)=\max\limits_{t>0} \widetilde{J}_{s, p}(t u)$.
\end{lm}

\begin{proof} 
Let $u\in H_{s,p}\backslash\{0\}$ be fixed and $t>0$. By similar arguments to that of (\ref{mc}), we have
\begin{eqnarray*}
\int_{\mathbb{Z}^d}(R_\alpha \ast F(tu^+))F(tu^+)\,d\mu\leq\varepsilon t^{2p}\|u\|_{H_{s,p}}^{2p}+C_\varepsilon t^{2\tau}\|u\|_{H_{s,p}}^{2\tau}.
\end{eqnarray*}
Then we have
$$
\begin{aligned}
\widetilde{J}(tu) & =\frac{t^{p}}{p}\|u\|_{H_{s,p}}^{p}-\frac{1}{2}\int_{\mathbb{Z}^d}(R_\alpha \ast F(tu^+))F(tu^+)\,d\mu \\&\geq\frac{t^{p}}{p}\|u\|_{H_{s,p}}^{p}-\varepsilon t^{2p}\|u\|_{H_{s,p}}^{2p}-C_\varepsilon t^{2\tau}\|u\|_{H_{s,p}}^{2\tau}.
\end{aligned}
$$
Since $2\tau>p$, let $\varepsilon\rightarrow 0^+$, one gets easily that $J_{s,p}(tu)>0$ for $t>0$ sufficiently small.

Moreover, by (\ref{md}), we have that
$$
 \lim _{t\rightarrow\infty} \widetilde{J}(tu)\rightarrow-\infty, \quad t \rightarrow \infty.  $$
Therefore, $\max\limits_{t>0} \widetilde{J}(tu)$ is achieved at some $t_{u}>0$ with $t_{u} u \in \mathcal{M}_{s,p}$.

In the following, we show the uniqueness of $t_{u}$. If there exist $t_{u}^{\prime}>t_{u}>0$ such that $t_{u}^{\prime} u\in\mathcal{M}_{s,p}$ and $t_{u} u \in \mathcal{M}_{s,p}$, then we have
$$
\begin{aligned}
& \|u\|_{H_{s,p}}^p=\int_{\mathbb{Z}^d}\frac{(R_\alpha \ast F(t'_uu^+))f(t'_uu^+)u^+}{\left(t_{u}^{\prime}\right)^{p-1}}\,d\mu, \\
& \|u\|_{H_{s,p}}^p=\int_{\mathbb{Z}^d}\frac{(R_\alpha \ast F(t_uu^+))f(t_uu^+)u^+}{\left(t_{u}\right)^{p-1}}\,d\mu,
\end{aligned}
$$
and it follows from $(f_4)$ that
$$
\begin{aligned}
0=\int_{\mathbb{Z}^d}\frac{(R_\alpha \ast F(t'_uu^+))f(t'_uu^+)u^+}{\left(t_{u}^{\prime}\right)^{p-1}}\,d\mu-\int_{\mathbb{Z}^d}\frac{(R_\alpha \ast F(t_uu^+))f(t_uu^+)u^+}{\left(t_{u}\right)^{p-1}}\,d\mu>0.
\end{aligned}
$$
This is impossible, hence there exists a unique $t_u>0$ such that $t_u u\in\mathcal{M}_{s,p}$ and $\widetilde{J}_{s,p}(t_{u} u)=$ $\max\limits_{t>0} \widetilde{J}_{s,p}(tu)$. 

If $u \in \mathcal{M}_{s, p}$, it is obvious that $t_{u}=1$. Therefore, we get that $\widetilde{J}_{s, p}(u)=\max\limits_{t>0} \widetilde{J}_{s, p}(t u)$.
\end{proof}

Secondly, we prove that $c_{s, p}$ can be achieved on the Nehari manifold $\mathcal{M}_{s, p}$.

\begin{lm}\label{i3}
There exists a function $u_{0} \in \mathcal{M}_{s, p}$ such that $\widetilde{J}_{s, p}\left(u_{0}\right)=c_{s, p}>0$.    
\end{lm}

\begin{proof}
Since $c_{s,p}=\inf\limits_{u\in \mathcal{M}_{s,p}} \widetilde{J}_{s,p}(u),$ there exists a minimizing sequence $\{u_{n}\} \subseteq \mathcal{M}_{s, p}$ such that $$\widetilde{J}_{s,p}(u_{n})\rightarrow c_{s, p}, \quad n \rightarrow\infty. $$ Moreover, we have $\langle \widetilde{J}_{s,p}^{\prime}(u_n), u_n\rangle=0$. Then similar to (\ref{ba}), we derive that $\{u_{n}\}$ is bounded in $H_{s, p}$. By Lemma \ref{lgg}, there exists $u_0 \in H_{s,p}$ such that
\begin{equation*}\label{av}
 u_n\rightharpoonup u_0,\quad \text{in}~H_{s,p},\qquad u_n\rightarrow u_0,\quad \text{strongly}~\text{in}~\ell^q(\mathbb{Z}^d), ~q\geq p.   
\end{equation*}
Clearly, $\{u^+_n\}$ is also bounded in $H_{s,p}$. By Lemma \ref{lo}, we have that 
\begin{equation}\label{bz}
\lim\limits_{n\rightarrow\infty} \int_{\mathbb{Z}^d}(R_\alpha \ast F(u^+_n))F(u^+_n)\,d\mu=\int_{\mathbb{Z}^d}(R_\alpha \ast F(u^+_0))F(u^+_0)\,d\mu,
\end{equation}
and
\begin{equation}\label{bl}
 \lim\limits_{n\rightarrow\infty}\int_{\mathbb{Z}^d}(R_\alpha \ast F(u^+_n))f(u^+_n)u^+_n\,d\mu=\int_{\mathbb{Z}^d}(R_\alpha \ast F(u^+_0))f(u^+_0)u^+_0\,d\mu.   
\end{equation}
Then by (\ref{bz}), we deduce that
\begin{equation}\label{vm}
 \begin{aligned}
    & \widetilde{J}_{s,p}(u_0)\\ =& \frac{1}{p}\|u_0\|^p_{H_{s,p}}-\frac{1}{2}\int_{\mathbb{Z}^d}(R_\alpha \ast F(u^+_0))F(u^+_0)\,d\mu\\ \leq & \frac{1}{p}\limsup _{n \rightarrow\infty} \|u_{n}\|_{H_{s, p}}^{p}-\frac{1}{2}\int_{\mathbb{Z}^d}(R_\alpha \ast F(u^+_0))F(u^+_0)\,d\mu  \\ = & \limsup _{n \rightarrow\infty}\left[\frac{1}{p}\|u_{n}\|_{H_{s, p}}^{p}-\frac{1}{2}\int_{\mathbb{Z}^d}(R_\alpha \ast F(u^+_n))F(u^+_n)\,d\mu\right]\\=&\limsup _{n \rightarrow\infty} \widetilde{J}_{s,p}(u_n)\\=& c_{s,p}.
\end{aligned}   
\end{equation}
Note that $u_n\in\mathcal{M}_{s,p}$. On the one hand, by Lemma \ref{lj}, we have
\begin{equation}\label{az}
 \widetilde{J}_{s,p}\left(u_{n}\right)=\max _{t>0} \widetilde{J}_{s,p}\left(t u_{n}\right).   
\end{equation}
On the other hand,
by (\ref{bl}), we get that
\begin{equation}\label{at}
\left\|u_{0}\right\|_{H_{s, p}}^{p} \leq \limsup _{n \rightarrow\infty}\left\|u_{n}\right\|_{H_{s, p}}^{p}=\limsup _{n \rightarrow\infty} \int_{\mathbb{Z}^d}(R_\alpha \ast F(u^+_n))f(u^+_n)u^+_n\,d\mu=\int_{\mathbb{Z}^d}(R_\alpha \ast F(u^+_0))f(u^+_0)u^+_0\,d\mu. 
\end{equation}
Moreover, by Lemma \ref{lu}, we have that $\|u_n\|_{H_{s,p}}\geq \eta>0$, which implies that $u_0\neq 0$. 

Next we show that $u_0\in\mathcal{M}_{s,p}$.  If not, (\ref{at}) implies that $\left\|u_{0}\right\|_{H_{s, p}}^{p}<\int_{\mathbb{Z}^d}(R_\alpha \ast F(u^+_0))f(u^+_0)u^+_0\,d\mu.$ By the fact $u_0\neq 0$ and Lemma \ref{lj}, there exists a unique $t_{0}>0$ such that $t_{0} u_{0} \in \mathcal{M}_{s, p}$,
which says that $\widetilde{J}_{s,p}\left(t_{0} u_{0}\right) \geq c_{s, p}$. Then it follows from (\ref{bz}), (\ref{bl})  and (\ref{az}) that
$$
\begin{aligned}
c_{s, p} \leq &\widetilde{J}_{s,p}\left(t_{0} u_{0}\right)\\=&\frac{t^p_0}{p}\|u_0\|^p_{H_{s,p}}
-\frac{1}{2} \int_{\mathbb{Z}^d}(R_\alpha \ast F(t_0u^+_0))F(t_0u^+_0)\,d\mu\\ <&\frac{t^p_0}{p}\int_{\mathbb{Z}^d}(R_\alpha \ast F(u^+_0))f(u^+_0)u^+_0\,d\mu-\frac{1}{2} \int_{\mathbb{Z}^d}(R_\alpha \ast F(t_0u^+_0))F(t_0u^+_0)\,d\mu\\ =&\lim\limits_{n\rightarrow\infty}\left(\frac{t^p_0}{p}\int_{\mathbb{Z}^d}(R_\alpha \ast F(u^+_n))f(u^+_n)u^+_n\,d\mu-\frac{1}{2} \int_{\mathbb{Z}^d}(R_\alpha \ast F(t_0u^+_n))F(t_0u^+_n)\,d\mu\right)\\ =&\lim\limits_{n\rightarrow\infty}\left(\frac{t^p_0}{p}\|u_n\|^p_{H_{s,p}}-\frac{1}{2} \int_{\mathbb{Z}^d}(R_\alpha \ast F(t_0u^+_n))F(t_0u^+_n)\,d\mu\right)\\ =&\lim\limits_{n\rightarrow\infty}\left(\frac{1}{p}\|t_0u_n\|^p_{H_{s,p}}-\frac{1}{2} \int_{\mathbb{Z}^d}(R_\alpha \ast F(t_0u^+_n))F(t_0u^+_n)\,d\mu\right)\\=&\lim\limits_{n\rightarrow\infty} \widetilde{J}_{s,p}(t_0u_n)\\ \leq &\lim\limits_{n\rightarrow\infty} \widetilde{J}_{s,p}(u_n)\\=& c_{s,p},
\end{aligned}
$$
which is a contradiction. Hence $u_{0} \in \mathcal{M}_{s, p}$ and  
$\widetilde{J}_{s,p}\left(u_{0}\right) \geq c_{s, p}$. By Combining  (\ref{vm}) and Lemma \ref{ld}, we derive that $\widetilde{J}_{s,p}\left(u_{0}\right)=c_{s, p}>0$.

\end{proof}

By Lemma \ref{lj} and Lemma \ref{i3}, we obtain that there exists a function $u_{0} \in \mathcal{M}_{s, p}$ such that
\begin{equation}\label{v0}
\max _{t>0} \widetilde{J}_{s,p}\left(t u_{0}\right)=\widetilde{J}_{s,p}\left(u_{0}\right)=c_{s, p}>0 .
\end{equation}
In the following, we prove that $u_{0}$ is a critical point of $\widetilde{J}_{s,p}$.

\begin{lm}\label{lf}
 There holds $\widetilde{J}_{s,p}^{\prime}\left(u_{0}\right)=0$.   
\end{lm}

\begin{proof}
By contradiction,  suppose that $\widetilde{J}_{s,p}^{\prime}\left(u_{0}\right) \neq 0$, then there exists $v \in H_{s, p} \backslash\{0\}$ such that $\langle \widetilde{J}_{s,p}^{\prime}(u_{0}), v\rangle \neq 0$. Without loss of generality, let $\langle \widetilde{J}_{s,p}^{\prime}(u_{0}), v\rangle<0$.  For $t, r\in\mathbb{R}$, we define
$$
\phi(t, r)=t u_{0}+r v.
$$
A direct calculation yields that
$$
\begin{aligned}
& \lim _{(t, r) \rightarrow(1,0)} \frac{\partial}{\partial r} \widetilde{J}_{s,p}(\phi(t, r))\\=&\lim _{(t, r) \rightarrow(1,0)} \frac{\partial}{\partial r} \left[\frac{1}{p}\|tu_0+rv\|^p_{H_{s,p}}-\frac{1}{2}\int_{\mathbb{Z}^d}(R_\alpha \ast F((tu_0+rv)^+))F((tu_0+rv)^+)\,d\mu\right]\\=&\int_{\mathbb{Z}^d}\left(|\nabla^s u_0|^{p-2}\nabla^s u_0 \nabla^s v+h(x)|u_0|^{p-2} u_0 v\right) \,d\mu-\int_{\mathbb{Z}^d}(R_\alpha \ast F(u^+_0))f(u^+_0)v\,d\mu\\=&
\left\langle \widetilde{J}_{s,p}^{\prime}\left(u_{0}\right), v\right\rangle\\<& 0.   
\end{aligned}
$$
Then there exist $0<\varepsilon_{1}<1$ and $\varepsilon_{2}>0$ small enough such that $\widetilde{J}_{s,p}(\phi(t, r))$ is strictly decreasing with respect to $r \in\left[-\varepsilon_{2}, \varepsilon_{2}\right]$ for any fixed $t \in\left[1-\varepsilon_{1}, 1+\varepsilon_{1}\right]$. For $r\in (0, \varepsilon_{2}]$, it follows from (\ref{v0}) that
\begin{equation}\label{v1}
\widetilde{J}_{s,p}(\phi(t, r))<\widetilde{J}_{s,p}(\phi(t, 0))=\widetilde{J}_{s,p}\left(t u_{0}\right) \leq c_{s, p}.
\end{equation}
For $\phi(t, r)=t u_{0}+r v,$ we define  \begin{equation}\label{v3}
\varphi(t, r)=\|\phi\|_{H_{s, p}}^{p}-\int_{\mathbb{Z}^d}(R_\alpha \ast F(\phi^+))f(\phi^+)\phi^+\,d\mu.    
\end{equation}
It is clear that $\varphi(t, 0)=t^{p} \psi(t)$ for $t>0$, where
$$
\psi(t)=\left\|u_{0}\right\|_{H_{s, p}}^{p}-\frac{\int_{\mathbb{Z}^d}(R_\alpha \ast F(tu^+_0))f(tu^+_0)u^+_0\,d\mu }{t^{p-1}}.
$$
By ($f_4$), one gets that the function $\psi(t)$ is strictly decreasing in $t>0$, and hence $\varphi(t,0)$. Moreover, note that $u_{0} \in \mathcal{M}_{s, p}$, we deduce that
$$
\varphi\left(1+\varepsilon_{1}, 0\right)<\varphi(1,0)=0<\varphi\left(1-\varepsilon_{1}, 0\right) .
$$
By the continuity of $\varphi(t, r)$ in $r \in \mathbb{R}$, there exists $\varepsilon_{3} \in\left(0, \varepsilon_{2}\right)$ such that
$$
\varphi\left(1+\varepsilon_{1}, r\right)<0<\varphi\left(1-\varepsilon_{1}, r\right), \quad r\in\left(0, \varepsilon_{3}\right] .
$$
Note that $\varphi(t, r)$ is also continuous in $t \in \mathbb{R}$, then for any $r \in\left(0, \varepsilon_{3}\right]$, there exists $t_{r} \in\left(1-\varepsilon_{1}, 1+\varepsilon_{1}\right)$ such that $\varphi\left(t_{r}, r\right)=0$. Then it follows from (\ref{v3}) that $\phi\left(t_{r}, r\right) \in \mathcal{M}_{s, p}$. Combined with $(\ref{v1})$, we obtain that
$$
c_{s, p}=\inf _{u \in \mathcal{M}_{s, p}} \widetilde{J}_{s,p}(u) \leq \widetilde{J}_{s,p}\left(\phi\left(t_{r}, r\right)\right)<c_{s, p},
$$
which is impossible. Hence we get that $\widetilde{J}_{s,p}^{\prime}\left(u_{0}\right) =0$.

\end{proof}

{\bf Theorem \ref{t1}.}
Lemma \ref{lj}, Lemma \ref{i3}, Lemma \ref{lf} and Lemma \ref{i0} ensure that  $u_{0}$ is a positive ground state solution to the equation (\ref{aa}). The proof is completed.

\qed

\
\

\section{Acknowledgements}
The author thanks the referees for  helpful comments and suggestions on this paper. Moreover, the author is supported by the National Natural Science Foundation of China, no.12401135.

\
\

{\bf Declarations}

\
\

{\bf Conflict of interest:} The author declares that there are no conflicts of interests regarding the publication of
this paper.

\end{document}